\documentclass[11pt]{amsart}
\usepackage{amssymb}
\usepackage{amsmath}

\usepackage{color}

\newtheorem{theorem}{Theorem}[section]
\newtheorem{lemma}[theorem]{Lemma}
\newtheorem{proposition}[theorem]{Proposition}
\newtheorem{corollary}[theorem]{Corollary}
\theoremstyle{definition}
\newtheorem{definition}[theorem]{Definition}
\newtheorem{example}[theorem]{Example}

\theoremstyle{remark}
\newtheorem{remark}[theorem]{Remark}
\numberwithin{equation}{section}

\def\id{\mathop{\rm id}}

\def\id{{\bf 1}\!\!{\rm I}}

\def\bc{{\mathbb C}}

\def\bn{{\mathbb N}}

\def\br{{\mathbb R}}

\def\l{\lambda}

\def\a{\alpha}

\def\b{\beta}

\def\d{\delta}
\def\g{\gamma}

\def\cf{{\mathcal F}}

\def\cf{\mathcal{ F}}

\def\ck{\mathcal{K}}

\def\cw{\mathcal {W}}

\begin{document}
\setcounter{page}{1}

\title[perturbation bounds]{Perturbation bounds of Markov semigroups on abstract states spaces}

\author[Nazife Erkur\c{s}un-\"Ozcan, Farrukh Mukhamedov]{Nazife Erkur\c{s}un-\"Ozcan$^1$ and Farrukh Mukhamedov$^2$$^*$}

\address{$^{1}$ Department of Mathematics, Faculty of Science, Hacettepe University, Ankara, 06800,Turkey.}
\email{{erkursun.ozcan@hacettepe.edu.tr}}

\address{$^{2}$ Mathematical Sciences Department, College of Science, United Arab Emirates University 15551, Al-Ain United
Arab Emirates.} \email{{far75m@yandex.ru; farrukh.m@uaeu.ac.ae}}




\date{Received: xxxxxx; Revised: yyyyyy; Accepted: zzzzzz.
\newline \indent $^{*}$ Corresponding author}

\begin{abstract}

In order to successfully explore quantum systems which are perturbations of simple models, it is essential to understand the complexity of perturbation bounds. We must ask ourselves: How quantum many-body systems can be artificially engineered to produce the needed behavior. Therefore, it is convenient to make use of abstract framework to better understand classical and quantum systems. Thus, our investigation's purpose is to explore stability and perturbation bounds of positive $C_0$-semigroups on abstract state spaces using the Dobrushin's ergodicity coefficient. Consequently, we obtain a linear relation between the stability of the semigroup and the sensitivity of its fixed point with respect to perturbations of $C_0$-Markov semigroups. Our investigation leads to the discovery of perturbation bounds for the time averages of uniform asymptotically stable semigroups. 
A noteworthy mention is that we also prove the equivalence of uniform and weak ergodicities of the time averages Markov operators in terms of the ergodicity coefficient, which shines a new light onto this specific topic.
Lastly, in terms of weighted averages, unique ergodicity of semigroups is also studied. Emphasis is put on the newly obtained results which is a new discovery in classical and non-commutative settings.  
\end{abstract} \maketitle

\section{Introduction}

Constructing simplified models is the main focus to successfully capture the underlying physics of real materials, appropriate in the circumstance where explanation of their physical properties and behavior are important. For example, in quantum information theory, it is essential to find the applications of constructing simplified models by the artificially engineered quantum many-body systems to produce a specific behavior. While studying theoretical models of many-body physics, it is crucial that the properties of the model are stable under perturbations to the model
itself. If the physical predictions of a model undergo dramatic
changes when the local interactions are modified by a small amount,
it is difficult to argue that the idealized model captures the
correct physics of the real physical system. Similarly, if the
correct behavior of an engineered quantum system relies on
infinitely precise control of all the local interactions, the
proposal will not be of much practical use. Thus, it is critical to go beyond stability of closed systems, and derive stability results for
open, dissipative systems (see\cite{LC,PC,SW}) in justifying theoretical models of real, noisy physical many-body
systems. Simultaneously, in the new proposals as well, for exploiting dissipation to
carry out quantum information processing tasks. These kinds of
investigations lead to perturbation theory for quantum dynamical
systems defined on some operator algebras. Besides this, current
literature gives plenty of very recent examples of applications of
Markov chain perturbation bounds, including computational statistics
and climate science, see \cite{CRS,CLMP,KE, SN,  SW1,WGCE}. We point out that the main results of the present paper
also could be applied to the mentioned Markov chains by choosing an appropriate state space. Moreover,
more accurate perturbation bounds for them is obtainable.

If we consider a classical Markov process on a general state space $(E,\cf,\mu)$ given by
the transition probabilities $P(t,x,A)$, then one can define a semigroup  as follows: $ T_tf(x)=\int f(y)P(t,x,dy)$, which acts on
$L^1(E,\mu)$-space (see \cite{K} for more details).  Hence, the theory of perturbations \cite{Kar,Mit} can be reformulated in terms of the defined semigroup defined on $L^1(E,\mu)$-space which is an ordered Banach space (lattice)  w.r.t. the standard ordering. There are many 
papers devoted to the investigation of stability and ergodic properties of semigroups defined on Banach lattices \cite{B,K,LM}. 

On the other hand, it is convinient and important to study several properties of physical and
probabilistic processes in abstract framework due to the classical and quantum cases
confined to this scheme(see \cite{Alf}) where we can observe some aspects and applications put to work in \cite{KE,RKW}. It is noteworthy to mention that in this abstract scheme, one considers an ordered normed spaces and
mappings of these spaces (see \cite{Alf,E}). Moreover, in this
setting, certain ergodic properties of Markov operators are
considered and investigated in \cite{B,BR,CRS,EW1,SZ}.

Our purpose is to investigate stability and perturbation bounds of
positive $C_0$-semigroups defined on abstract state spaces (i.e. ordered Banach spaces). 
Roughly speaking, an abstract state space is an ordered Banach space with additivity property on the cone of positive elements. Such spaces
contain all classical $L^1$-spaces and the space of density operators acting on some Hilbert space \cite{Alf,Jam}. 
We refer the reader to \cite{BR} (see also
\cite{ALM}) for more information about the structure of positive
semigroups defined ordered spaces. In the investigation of $C_0$-semigroups it is important to employ methods of 
spectral analysis of operators.This spectral approach can be found in the standard textbooks as \cite{HP} . On non-spectral analysis of asymptotic behavior of positive semigroups especially on semigroups of positive operators on ordered Banach spaces is given in \cite{E}. 

We notice that averages of operators appear in many fields of mathematics, in theoretical physics and in other areas in science.  It is a natural question to ask its convergence behavior with respect to various kinds of convergence. In the present paper, we are going to investigate perturbation bounds of averages of $C_0$-semigroup of Markov operators defined on abstract state spaces. 
Basically, our main tool to study
asymptotical stability of the semigroups will be so-called
Dobrushin's ergodicity coefficient.  We notice that the mentioned
coefficient has been introduced in \cite{M0,M01} for positive
mappings defined on base normed spaces (or abstract
state spaces).  In \cite{EM2017} we have announced several perturbation results for single Markov operator on ordered Banach spaces. On the various convergence theorems, to determine the conditions is not easy and also many important operators can not be directly used. One of the ways is to construct the given semigroups from simpler ones. Perturbation is the standard methods for this approach. Therefore, it is important to determine a class of operators that is $A+B$ which is still the generator of a semigroup. A class of operators admissible in this sense is the class of bounded operators.
Since, the coefficient is one of the effective tools
in investigating limiting behavior of classical Markov processes (see
\cite{IS,Se2} for review), therefore, in the present paper, we are going to
study one of the questions that ask how sensitive stationary states are related
to perturbations of uniform asymptotical stability (i.e. rapid
mixing in physical terminology) of Markov semigroups.

The paper is organized as follows. In Section 3, we first establish
uniform asymptotical stability for $C_0$-Markov semigroups (defined on abstract state spaces) in terms
of Dobrushin's ergodicity coefficient. Consequently, we
obtain a linear relation between the stability of the semigroup and
the sensitivity of its fixed point with respect to perturbations of
Markov operators. Moreover, we establish perturbation bounds
for the time averages of the uniform asymptotically stable
semigroups. In Section 4, we go on to prove the equivalence of
uniform and weak ergodicities of the time averages Markov operators
in terms of the ergodicity coefficient which is a new insight to
this topic. This result allows us to produce certain perturbation
bounds for such kind of semigroups. In the last section 5, we study unique ergodicity and its relation to the convergence of weighted
averages. Note that the considered Banach space contains as a
particular case of several interesting operator spaces such as von
Neumann algebras, Jordan algebras. Therefore, our results can be
applied to semigroups acting on these spaces which have many
applications in quantum theory \cite{RKW}. We stress that when the
von Neumann algebra is taken as a matrix algebra, some particular
cases of the present paper were studied in \cite{SW}. Moreover, the
results of the last sections even new in the classical and quantum
settings.

We emphasize that in the present paper, we are going to consider a general abstract state space for which the convex hull of the base $\ck$ and $-\ck$ is not
assumed to be radially compact (in our previous papers \cite{EM2017,M0,M01} this condition was essential). This consideration has an important advantage: whenever $X$ is an ordered Banach space with 
a generating cone $X_+$ whose norm is additive on $X_+$, then $X$ admits an equivalent norm which coincides
with the original norm on $X_+$ and which renders $X$ such a base norm space. Hence, to apply the results of the paper
one would then only have to check that the norm is additive on $X_+$.  

\section{Preliminaries}

In this section we recall some necessary definitions and fact
about ordered Banach spaces.

 Let $X$ be an ordered vector space
with a cone $X_+=\{x\in X: \ x\geq 0\}$. A subset $\ck$ is called
a {\it base} for $X$, if one has $\ck=\{x\in X_+:\ f(x)=1\}$ for
some strictly positive (i.e. $f(x)>0$ for $x>0$) linear functional
$f$ on $X$. An ordered vector space $X$ with generating cone $X_+$
(i.e. $X=X_+-X_+$) and a fixed base $\ck$, defined by a functional
$f$, is called {\it an ordered vector space with a base}
\cite{Alf}. In what follows, we denote it as $(X,X_+,\ck,f)$. Let
$U$ be the convex hull of the set $\ck\cup(-\ck)$, and let
$$
\|x\|_{\ck}=\inf\{\l\in\br_+:\ x\in\l U\}.
$$
Then one can see that $\|\cdot\|_{\ck}$ is a seminorm on $X$.
Moreover, one has $\ck=\{x\in X_+: \ \|x\|_{\ck}=1\}$,
$f(x)=\|x\|_{\ck}$ for $x\in X_+$.  
Assume that the seminorm becomes a norm, and $X$ is complete w.r.t. this norm and $X_+$ is closed, then 
$(X,X_+,\ck,f)$ is called \textit{abstract state
space}. In this case, $\ck$ is a closed face of the unit ball of $X$,  and $U$ contains the open
unit ball of $X$. 
If the set $U$ is \textit{radially compact} \cite{Alf}, i.e.
$\ell\cap U$ is a closed and bounded segment for every line $\ell$
through the origin of $X$,  then  $\|\cdot\|_{\ck}$ is a
norm. The radially compactness is equivalent to $U$   coincides with the closed unit ball of $X$.  In this is case, 
we will call $X$ a \textit{strong abstract state
space}.  In the sequel, for the sake of simplicity instead of
$\|\cdot\|_{\ck}$ we will use usual notation $\|\cdot\|$.  We refer the reader to
\cite{Yo} where the difference between the strong ASS and  a more
general class of base norm spaces is discussed.

We recall some elementary definitions and
results. The positive cone $X_+$  of an ordered Banach space $X$  is said to be $\l$-generating if, given 
$x\in X$ , we can find $y,z\in X_+$ such that $x = y+ z$ and $\|y\| + \|z\|\leq\l \|x\|$. 
The norm on $X$ 
is  called \textit{regular} (respectively, \textit{strongly regular}) if, given $x$ in the open (respectively,
closed) unit ball of $X$, we can find $y$ in the closed unit ball with $y\geq  x$ and $y\geq -x$ . The
norm is said to be additive on $X_+$ if $\|x + y\| = \|x\| + \|y\|$ for all $ x, y\in X_+$.
We notice that if $X_+$ C is 1-generating, then one can show that $X$ is strongly regular. Similarly, if $X_+$ is
$\l$-generating for all $\l > 1$, then $X$ is regular \cite{Yo}.
The following results are well-known.

\begin{theorem}\cite{WN} Let  $X$ be an ordered Banach space with closed positive cone $X_+$.
The following statements are equivalent:
\begin{itemize}
\item[(i)] $X$ is an abstract state space;
\item[(ii)] $X$ is regular, and the norm is additive on $X_+$;
\item[(iii)] $X_+$ is $\l$-generating for all $\l > 1$, and the norm is additive on $X_+$.
\end{itemize}
\end{theorem} 

\begin{theorem}\cite{Yo} Let  $X$ be an ordered Banach space with closed positive cone $X_+$.
The following statements are equivalent:
\begin{itemize}
\item[(i)] $X$ is a strong abstract state space;
\item[(ii)] $X$ is strongly regular, and the norm is additive on $X_+$;
\item[(iii)] $X_+$ is 1-generating and the norm is additive on $X_+$.
\end{itemize}
\end{theorem}

Let $A$ be a real ordered linear space and as before $ A_+$ denotes
the set of positive elements of $A$. An element $e\in A_+$ is called
\textit{order unit} if for every $a\in A$ there exists a number
$\l\in\br_+$ such, that $-\l e\leq a\leq\l e$. If the order is
Archimedean then the mapping $a\to\|a\|_e=\inf\{\l> 0\ : \-\l e\leq
a\leq\l e\}$  is a norm. If $A$ is a Banach space with respect to
this norm, the the pair $(A, e)$ is called \textit{an order-unit
space with the order unit $e$}.

Let us provide some examples of
ASS.

\begin{itemize}

\item[1.] Let $M$ be a von Neumann algebra. Let $M_{h,*}$ be the
Hermitian part of the predual space $M_*$ of $M$. As a base $\ck$ we
define the set of normal states of $M$. Then
$(M_{h,*},M_{*,+},\ck,\id)$ is a strong ASS, where $M_{*,+}$ is the set of
all positive functionals taken from $M_*$, and $\id$ is the unit in
$M$.

\item[2.] Let $(A,e)$ be an order-unit
space. An element $\rho\in A^*$ is called \textit{positive} if
$\rho(x)\geq 0$ for all $a\in A_+$. By $A^*_+$ we denote the set of
all positive functionals. A positive linear functional is called a
\textit{state} if $\rho(e)=1$. The set of all states is denoted by
$S(A)$. Then it is well-known \cite{Alf} that $(A^*,A^*_+,S(A),e)$
is a strong ASS.

\item[3.] Let $X$ be a Banach space over $\br$. Consider a new Banach space
$\tilde X=\br\oplus X$ with a norm $\|(\a,x)\|=\max\{|\a|,\|x\|\}$.
Define a cone $\tilde X_+=\{(\a,x)\ : \ \|x\|\leq \a, \
\a\in\br_+\}$ and a positive functional $f(\a,x)=\a$. Then one can
define a base $\ck=\{(\a,x)\in\tilde X:\ f(\a,x)=1\}$. Clearly, we
have $\ck=\{(1,x):\ \|x\|\leq 1\}$. Then $(\tilde X,\tilde
X_+,\tilde \ck,f)$ is an abstract state base-norm space \cite{Jam}. Moreover, $X$
can be isometrically embedded into $\tilde X$.  Using this construction one can study several interesting examples of ASS. In particularly, one considers as $X=\ell_p$, $1<p<\infty$. 

\item[4.]  Let $A$ be the disc algebra, i.e. the sup-normed space of complex-valued
functions which are continuous on the closed unit disc, and analytic on the open unit disc.
Let $X =\{f\in A :\ f(1)\in\br \}$. Then $X$ is a real Banach space with the following positive cone
$X_+=\{f\in X: f(1)=\|f\|\}=\{f\in X: \ f(1)\geq\|f\|\}$. The space $X$ is ASS, not strong ASS (see  \cite{Yo} for details).
\end{itemize}

Let $(X,X_+,\ck,f)$ be an ASS. A linear operator $T:X\to X$ is
called \textit{positive}, if $Tx\geq 0$ whenever $x\geq 0$. A positive linear
operator $T:X\to X$ is said to be {\it Markov}, if $T(\ck)\subset\ck$.
It is clear that $\|T\|=1$, and its adjoint mapping $T^*: X^*\to
X^*$ acts in ordered Banach space $X^*$ with unit $f$, and moreover,
one has $T^*f=f$. Note that in case of $X=\br^n$, $X_+=\br_+^n$ and
$\ck=\{(x_i)\in\br^n: \ x_i\geq 0, \ \sum_{i=1}^n x_i=1\}$, then for
any Markov operator $T$ acting on $\br^n$, the conjugate operator
$T^*$ can be identified with a usual stochastic matrix. Now for each
$y\in X$ we define a linear operator $T_y: X\to X$ by
$T_y(x)=f(x)y$.

The main object in this paper is a strongly continuous semigroup or \textit{$C_0$-semigroup} on a Banach space $ X $ is a map $ {\displaystyle T_t :\mathbb {R} _{+}\to L(X)}$ such that
\begin{itemize}
\item[(i)] ${\displaystyle T_0=I} $,   identity operator on ${\displaystyle X}$;
\item[(ii)] ${\displaystyle \forall t,s\geq 0:\ T_{t+s}=T_tT_s} $;
\item[(iii)] ${\displaystyle \forall x_{0}\in X:\ \|T_t x_{0}-x_{0}\|\to 0}$, as ${\displaystyle t\downarrow 0} $.
\end{itemize}
The first two axioms are algebraic, and state that ${\mathcal T} =
(T_t) $ is a representation of the semigroup ${\displaystyle
{(\mathbb {R} _{+},+)}} $; the last is topological, and states that
the map ${\mathcal T}$ is continuous in the strong operator
topology.

Let  $\mathcal T=(T_t)_{t\geq 0}$ be a $C_0$-semigroup. If for each
$t\in\br_+$ the operator $T_t$ is Markov, then  $\mathcal
T=(T_t)_{t\geq 0}$ is called a \textit{$C_0$-Markov semigroup}. An
element $x_0\in X$ is called a \textit{fixed point}  or \textit{stationary point} of  $\mathcal
T=(T_t)_{t\geq 0}$ if one has $T_tx_0=x_0$ for all $t\in\br_+$. In
\cite{BR} a basic structure of positive semigroups defined on
ordered Banach spaces is given. We refer the reader to \cite{CLMP,CRS,WGCE}, for more concrete examples of Markov semigroups associated with certain physical systems. 

We denote
$$
\displaystyle A_t (\mathcal T)=\frac{1}{t} \int_{0}^{t} T_s ds, \ \
t\in\br_+.
$$
We notice that the integral above is taken with respect to the strong operator topology. 

\begin{definition} A $\mathcal T=(T_t)_{t\geq 0}$ be a $C_0$-Markov semigroup defined on $X$ is called
\begin{enumerate}
\item[(i)] {\it uniform asymptotically stable} if there exist an
element $x_0\in\ck$ such that
$$
\lim_{t\to\infty}\|T_t - T_{x_0}\|=0;
$$
\item[(ii)] {\it uniform mean ergodic} if there exist an element
$x_0\in \ck$ such that
$$
\lim_{t\to\infty} \bigg\|A_t(\mathcal T)-T_{x_0}\bigg\|=0;
$$
\item[(iii)] \textit{weakly ergodic} if one has
$$
\lim_{t\to\infty}\sup_{x,y\in \ck}\|T_t x-T_t y\|=0;
$$
\item[(iv)] {\it weakly mean ergodic} if one has
$$
\lim_{t\to\infty} \sup_{x,y\in \ck} \|A_t(\mathcal T)x - A_t(\mathcal T)y\|=0.
$$
\end{enumerate}
\end{definition}

\begin{remark}\label{dd0} We notice that uniform asymptotical
stability implies uniform mean ergodicity. Moreover, if $\mathcal T$
is uniform mean ergodic, then $x_0$ corresponding to $T_{x_0}$, is
a fixed point of $\mathcal T$. Indeed,  $T_tT_{x_0}=T_{x_0}$, which
yields $T_tx_0=x_0$ for all $t\in\br_+$. We stress that every
uniformly mean ergodic Markov operator has a unique fixed point.
\end{remark}

Let $(X,X_+,\ck,f)$ be an ASS  and $T:X\to X$ be a Markov operator.
Letting
\begin{equation}
\label{NN} N=\{x\in X: \ f(x)=0\},
\end{equation}
we define
\begin{equation}
\label{db} \d(T)=\sup_{x\in N,\ x\neq 0}\frac{\|Tx\|}{\|x\|}.
\end{equation}

The quantity $\d(T)$ is called \textit{Dobrushin's ergodicity
coefficient} of $T$ (see \cite{M0}).

\begin{remark} We note that if $X^*$ is a commutative algebra, the notion of
the Dobrushin ergodicity coefficient was studied in
\cite{C},\cite{D}. In a non-commutative setting, i.e. when $X^*$
is a von Neumann algebra, such a notion was introduced in
\cite{M}. We should stress that such a coefficient has been
independently defined in \cite{GQ}. Furthermore, for particular
cases, i.e. in a non-commutative setting, such a coefficient
explicitly has been calculated for quantum channels (i.e.
completely positive maps). Moreover, its physical interpratations and applications to concrete quantum physical systems
can be found in \cite{CLMP,SN,WGCE}. In these settings, $X^*$ is taken as certain $C^*$-algebra associated with the system. 
\end{remark}

The next result establishes several properties of Dobrushin
ergodicity coefficient.

\begin{theorem}\label{Dob}
Let $(X, X_+ , \ck ,f)$ be an abstract state space and $T,S : X\to X$ be Markov
operators. The following statements hold:
\begin{itemize}
\item[(i)] $0 \leq \d (T) \leq 1$; \item[(ii)] $|\d (T) - \d (S)|
\leq \d (T-S) \leq \left\| T - S\right\|$; \item[(iii)] $\d (TS)
\leq \d (T) \d (S)$; \item[(iv)] if $H: X\to X$ is a linear bounded
operator such that $H^* (f) =0$, then $\left\| TH \right\| \leq \d
(T) \left\| H \right\|$; \item[(v)] 
If  $X_+$ is $\l$-generating, then one has
\begin{equation}\label{5Dob}
\displaystyle \d (T) \leq  \frac{\l}{2} \sup_{u,v \in \ck} \left\| Tu -
Tv\right\|;
\end{equation}
\item[(vi)] if $\d (T)=0$, then there exists $y_0 \in X_+$ such
that $T=T_{y_0}$.
\end{itemize}
\end{theorem}

\begin{proof} The statements (i)-(iv), (vi) have been proved in \cite{M0}. We only need to prove (v). Let $x\in N$. Then one can find
$x_+,x_-\in X_+$ such that $x=x_+-x_-$ with $\|x_+\|+\|x_-\|\leq\l\|x\|$. The equality $f(x)=0$ implies $f(x_+)=f(x_-)$ which means $\|x_+\|=\|x_-\|$. Hence, we have $\xi:=\|x_+\|=\|x_-\|\leq\frac{\l}{2}\|x\|$, so by denoting 
$$
u=\frac{1}{\|x_+\|}x_+, \ \ \ \ v=\frac{1}{\|x_-\|}x_-
$$
one finds $x=\xi(u-v)$. Consequently, one gets
\begin{eqnarray*}
\frac{\|Tx\|}{\|x\|}=\frac{\xi\|T(u-v)\|}{\|x\|}\leq \frac{\l}{2}\|Tu-Tv\|,
\end{eqnarray*} 
this together with \eqref{db} yields the required assertion.
\end{proof}

\begin{remark}\label{sASS} We point out that if $X$ is a strong abstract state space (i.e. $X_+$ is 1-generating), then an analogous fact as Theorem \ref{Dob} (v) has been proved in \cite{M0}, and in this case \eqref{5Dob} reduces to 
\begin{equation}\label{51Dob}
\d (T) = \frac{1}{2} \sup_{u,v \in \ck} \left\| Tu -
Tv\right\|.
\end{equation}
\end{remark}

\section{Uniform asymptotic stability Perturbation Bounds of $C_0$-Markov semigroups}

In this section, we prove uniform asymptotic stability of
$C_0$-Markov semigroups in terms of the Dobrushin's ergodicity
coefficient. This allows us to establish perturbation bounds for the
semigroups using the ergodicity coefficient. The following result is
an continuous time analogue of Theorem 3.5 \cite{M0}. For the sake
of completeness we will prove it.

\begin{theorem}\label{Dob1}
Let $(X, X_+ , \ck ,f)$ be an abstract state space and $\mathcal T =(T_t)_{t\geq 0}$
be a $C_0$-Markov semigroup on $X$. Then the following statements are
equivalent:
\begin{itemize}
\item[(i)] $\mathcal T$ is weakly ergodic; \item[(ii)] there exists $t_0 \in\br$ such that  $ \d (T_{t_0}) < 1$;
\item[(iii)] $\mathcal T$ is uniformly asymptotical stable.
 Moreover,
there are positive constants $C, \a, t_0 \in\br$ and $x_0 \in\ck$
such that
\begin{equation}\label{Dobrushin}
\left\| T_t - T_{x_0} \right\| \leq C e^{-\a t}, \,\,\,\,\, \forall
t\geq t_0.
\end{equation}
\end{itemize}
\end{theorem}

\begin{proof}
The implication  (i)$ \Rightarrow$ (ii) immediately follows from Theorem \ref{Dob}(v),  and (iii) $\Rightarrow$ (i)
is obvious. Therefore, it is enough to establish the implication
(ii)$ \Rightarrow$ (iii). Let  $t_0 \in \mathbb
R_+$ such that $\delta (T_{t_0})<1$. We put $\rho:=\delta (T_{t_0})$. Clearly,  $0<\rho<1$.  Since for every $t \in
\mathbb R$, $T_t$ is a Markov operator, from Theorem \ref{Dob} (iii)
we obtain
\begin{equation}\label{Dob11}
\delta (T_t) = \delta (T_{[t/t_0] t_0 + r}) \leq \delta
(T_{t_0})^{[t/t_0]} \leq \rho^{[t/t_0]} \to 0, \ \ \ \text{as} \ \ \
t\to \infty,
\end{equation}
where $[a]$ denotes the integer part of a number $a$ and $0\leq r <
t_0$.

Now let us show that $\mathcal T$ is a Cauchy net with respect to
the norm. Indeed, thanks to  Theorem \ref{Dob} (iv) and \eqref{Dob11}
one has
\begin{eqnarray}
\left\| T_t - T_{t+s} \right\| &=& \left\| T_{[t/t_0] t_0 + r} - T_{([t/t_0] t_0 + r) +s} \right\| \nonumber \\
&\leq& \delta (T_{t_0})^{[t/t_0]} \left\| T_r - T_{r + s}
\right\|\nonumber\\[2mm]
&\leq& 2\delta (T_{t_0})^{[t/t_0]}\to 0 \ \ \ \text{as} \ \ \ t\to
\infty.
\end{eqnarray}

Therefore, there is an operator $Q$ such that $\left\| T_t - Q
\right\| \to 0$. One can show that $Q$ is a Markov operator. To show
$Q = T_{x_0}$ for some $x_0 \in X_+$, it is enough to establish
$\delta (Q) =0$. Indeed, from
$$
|\delta (T_t) - \delta(Q) | \leq \left\| T_t - Q \right\| \ \
\text{and} \ \ \lim_{t\to\infty} \delta (T_t) =0,
$$
we infer that $\delta (Q) =0$. This completes the proof.
\end{proof}

\begin{remark}
From this theorem we immediately conclude that if at least one operator $T_{t_0}$
(for some $t_0\in\br_+$) of a $C_0$-Markov semigroup
$\mathcal T =(T_t)_{t\geq 0}$ is uniformly asymptotical stable,
then whole semigroup $\mathcal T$ is uniform asymptotically stable. We point out that in \cite{M0,M01} the space $X$ was strong ASS and the semigroup is taken to be discrete $\{T^n\}$. 
\end{remark}

\begin{remark}
In the classical setting, i.e. if $X=L^1(E,\mu)$, then a similar kind of result has been
established in \cite{Mit1}. 
\end{remark}

To establish perturbation bounds for $C_0$-semigroups, we need some
auxiliary well-known facts (see \cite{EN} for the proofs).

\begin{theorem}\label{gg1}
 Let $(A,D(A))$ be the generator of a $C_0$-semigroup $\mathcal T =(T_t)_{t\geq 0}$ on a Banach space $X$
satisfying
$$
\left\| T_t \right\|    \leq M e^{wt} \ \ \forall  t \geq 0
$$
and some $w \in \br, M \geq 1$. If $B \in L(X)$, then $C := A + B$
with $D(C) := D(A)$ generates a $C_0$-semigroup $\mathcal S =
(S_t)_{t\geq 0}$
 satisfying
$$
\left\| S_t\right\| \leq  M e^{(w+M \left\|B\right\|)t} \ \ \forall
t \geq 0.
$$
Moreover,  for a representation formula for this new semigroup, it
satisfies an integral equation.
$$
\displaystyle
 S_t x = T_t x + \int_0^t T_{t - s} B S_s xds
$$
holds for every $t \geq 0 $ and $x \in X$.
\end{theorem}

Out first result about perturbation bounds is the following result.

\begin{theorem} \label{per1}
Let $(X, X_+, \ck, f)$ be an abstract state space and
$\mathcal S=(S_t)_{t\geq 0}$, $\mathcal T=(T_t)_{t\geq 0}$ be
$C_0$-Markov semigroups on $X$ with generators $A$ and $C $,
respectively, satisfying $B:= C - A$ is bounded. If there exists
$\rho \in [0,1)$ and $t_0 \in \br$ such that $\d (T_ {t_0}) \leq
\rho$, then one has
\begin{eqnarray} \label{1a}
\displaystyle
&& \left\| T_t x- S_t z \right\| \leq \\
&& \begin{cases} \left\| x - z \right\| + t \left\| B \right\|,
&\forall t \leq t_0,\\ \displaystyle
 \rho^{[t/t_0]} \left\| x- z \right\| + \Big( \frac{t_0 (1 - \rho^{[t/t_0]})}{1-\rho} + \rho^{[t/t_0]} (t - t_0 [t/ t_0])\Big) \left\| B\right\|, &
\forall t > t_0
\end{cases} \nonumber
\end{eqnarray}
for every $x,z\in \ck$.
\end{theorem}

\begin{proof} From Theorem \ref{gg1}, for each $t\in \br$,  we have
\begin{equation} \label{2}
S_t x= T_t x+ \int_0^t T_{t - s} B S_s xds.
\end{equation}
Let $x,z\in \ck$, it then follows from \eqref{2} that
\begin{eqnarray*}
T_t x - S_t z &=& T_t x - T_t z - \int_0^t T_{t - s} B S_s zds \\
&=& T_t  (x - z) - \int_0^t T_{t-s} B S_s z ds \\
&& \text{using change of variable for t-s= s, then we have} \\
&=& T_t  (x - z) - \int_0^t T_s B S_{t-s} z ds \\
&=& T_t  (x - z) - \int_0^t T_{s} B z_s ds
\end{eqnarray*}
where $z_s:=S_{t-s} z$. Hence,
\begin{equation}\label{per11}
\left\| T_t x  - S_t z \right\| \leq \left\| T_t  (x - z) \right\| +
\int_0^t \left\|  T_s  B (z_s) \right\|.
\end{equation}

The Markovianity of $T_t$, $S_t$ and Theorem \ref{Dob} (iv)
imply that
$$
 \left\|  T_s \circ B (z_s) \right\| \leq \d(T_s) \left\| B\right\|,
 \ \ \ \
\left\| T_t (x - z) \right\| \leq \d (T_t) \left\| x -z\right\|.
$$
Hence, from \eqref{per11} we obtain
\begin{eqnarray} \label{3}
\displaystyle \left\| T_t x  - S_t z \right\| \leq \d (T_t) \left\|x
- z\right\| + \left\| B \right\| \int_{0}^{t} \d (T_s) ds.
\end{eqnarray}

Moreover if $t > t_0$ and $t / t_0$ is not an integer, then
$\delta (T_t) < 1$ implies that $\delta (T_t ) < \delta (T_{t_0} )^{
[t/t_0]} < \rho^{[t/t_0]}$.  So by Riemann integral,  we obtain
\begin{eqnarray} \label{4}
\displaystyle \int_{0}^{t} \d (T_s) ds &\leq& t_0 (1+ \d (T_{t_0})  + \d (T_{t_0})^2  + \cdots + \d (T_{t_0})^{[t/t_0] -1} ) + \d (T_{t_0})^{[t/t_0]} (t - t_0 [t/ t_0])\nonumber \\
&=& \frac{t_0 (1- \d (T_{t_0})^{[t/t_0] -1})}{1- \d (T_{t_0})} +  \d
(T_{t_0})^{[t/t_0]} (t - t_0 [t/ t_0]), \, \, \forall t > t_0 .
\end{eqnarray}

Hence, the last inequality with \eqref{3} yields the required
assertion. \end{proof}

\begin{remark}
Note that this result extends all existing results (see
\cite{Mit2,Mit3,SW}) for general spaces. In particularly, if we take
as $X$ a predual of any von Neumann algebra, then we extend the
result of \cite{SW} (where it was proved similar result for matrix
algebras) for arbitrary von Neumann algebras.
\end{remark}

\begin{corollary} \label{per2}
Let the conditions of Theorem \ref{per1} be satisfied. Then for
every $x,y\in\ck$ one has
\begin{equation} \label{5}
\displaystyle \sup_{t \geq 0} \left\| T_t x - S_t z \right\| \leq
\left\| x - z \right\| + \frac{t_0}{1- \rho} \left\| B\right\|.
\end{equation}
In addition, if $\mathcal S$ is uniformly asymptotical stable to
$S_{z_0}$ then
\begin{equation} \label{6}
\displaystyle \left\| T_{x_0} - S_{z_0} \right\| \leq \frac{t_0}{1-
\rho} \left\| B\right\|.
\end{equation}
\end{corollary}

\begin{proof}
The inequality \eqref{5} is a direct consequence of \eqref{1a}. Now
taking the limit as $t\to \infty$ in \eqref{3} we immediately find
$$
\left\| T_{x_0} - S_{z_0}\right\|\leq \left\| B \right\|
\int_{0}^{\infty} \d (T_s) ds.
$$
From \eqref{4} it follows that
$$
\left\| T_{x_0} - S_{z_0} \right\| \leq \left\| B \right\|
\frac{t_0}{1-\rho} .
$$
This completes the proof.
\end{proof}

The following theorem gives an alternative method of obtaining
perturbation bounds in terms of $\d (T_{t_0})$.

\begin{theorem} \label{per4}
Let the conditions of Theorem \ref{per1} be satisfied.  Then for
every $x,z\in\ck$ one has

\begin{eqnarray} \label{12}
\left\| T_t x - S_t z \right\| &\leq & \d (T_{t_0})^{\lfloor t/t_0
\rfloor} (\left\| x - z\right\| + \sup_{0< t< t_0} \left\| T_t -
S_t \right\|) \\[2mm]
&&+ \frac{1 - \d (T_{t_0})^{\lfloor t/t_0\rfloor}}{1 - \d (T_{t_0})}
\left\| T_{t_0} - S_{t_0}\right\|, \,\,\,\, t \in \br_+ \nonumber.
\end{eqnarray}
\end{theorem}

\begin{proof} If $t < t_0$ then by the inequality
\begin{eqnarray*}
\left\| T_t x - S_t z \right\| &\leq& \left\| T_t x - S_t x \right\| + \left\| S_t (x-z)\right\| \\
&\leq& \left\| x-z \right\| + \left\| T_t - S_t \right\|,
\end{eqnarray*}
we have \eqref{12}.

If $t\geq t_0 $, we obtain
\begin{eqnarray*}
T_t x - S_t z &=& T_{t_0} (T_{t - t_0} x) - S_{t_0} (S_{t-t_0} z) \\
&=& T_{t_0} (T_{t-t_0} x - S_{t-t_0} z) + (T_{t_0} - S_{t_0})
S_{t-t_0} z.
\end{eqnarray*}
Therefore,
\begin{eqnarray*}
\left\| T_t x - S_t z \right\| \leq \left\| T_{t-t_0} x - S_{t-t_0}
z\right\| \d (T_{t_0}) + \left\| T_{t_0} - S_{t_0} \right\|.
\end{eqnarray*}

If we continue to apply this relation to
$$\left\| T_{t-t_0} x - S_{t-t_0} z\right\|, \cdots,  \left\| T_{t- t_0(\lfloor t/t_0\rfloor - 1)} x - S_{t- t_0(\lfloor t/t_0\rfloor - 1)} z \right\|
$$
we obtain
\begin{eqnarray*} \label{13}
\left\| T_t x - S_t z \right\| &\leq& \d (T_{t_0})^{\lfloor t/t_0
\rfloor} (\left\| x - z\right\| +  \sup_{0< t< t_0} \left\| T_t -
S_t \right\|)
\nonumber\\[2mm]
 && + \bigg(\d (T_{t_0})^{\lfloor t/t_0 \rfloor - 1} + \d
(T_{t_0})^{\lfloor t/t_0 \rfloor - 2} + \cdots + 1\bigg) \left\| T_{t_0} - S_{t_0} \right\| \nonumber, \\
&= & \d (T_{t_0})^{\lfloor t/t_0 \rfloor} (\left\| x - z\right\| +
\max_{0< t< t_0} \left\| T_t - S_t \right\|)\\[2mm]
&& + \frac{ 1- \d (T_{t_0})^{\lfloor t/t_0 \rfloor}}{1-\d (T_{t_0})}
\left\| T_{t_0} - S_{t_0} \right\|.
\end{eqnarray*}
\end{proof}

\begin{theorem} \label{per6}
If $\d (T_{t_0}) < 1$ for some $t_0\in\br_+$, then every
$C_0$-Markov semigroup $\mathcal S= (S_t)_{t\geq 0}$ satisfying
$\|S_{t_0} -T_{t_0}\|<1-\d (T_{t_0})$ is uniformly asymptotical
stable and has a unique fixed point $z_0\in\ck$ such that
\begin{eqnarray} \label{per62}
\|x_0-z_0\|\leq \frac{\|S_{t_0}-T_{t_0}\|}{1-\d
(T_{t_0})-\|S_{t_0}-T_{t_0}\|}
\end{eqnarray}
\end{theorem}

\begin{proof}  Take any $x\in
N$, then we have
\begin{equation}\label{per61}
\|S_{t_0}x\|\leq\|S_{t_0}-T_{t_0} x\|+\|T_{t_0}x\|\leq \rho\|x\|.
\end{equation}
where $\rho=\|S_{t_0}-T_{t_0}\|+\d(T_{t_0})<1$. Hence by
\eqref{per61} one gets $\|S_{t_0}^{n}x\|\leq \rho^n\|x\|$ for all
$n\in\bn$. Hence $(I - S_{t_0})$ is invertible on  $N$.

It is clear that the equation $S_{t_0} z_0=z_0$ with $z_0\in\ck$
such that $(I-S_{t_0})(z_0-x_0)=-(I-S_{t_0})x_0$. Since
$(I-S_{t_0})x_0$ is an element of $ N$, we conclude that
$$
z_0=x_0-(I-S_{t_0})^{-1}((I-S_{t_0})x_0)
$$
is unique. Moreover from the identity
$$
z_0-x_0=T_{t_0}(z_0-x_0)+(S_{t_0}-T_{t_0})(z_0-x_0)+(S_{t_0}-T_{t_0})x_0
$$
and
$$
\|z_0-x_0\|\leq
\big(\d(T_{t_0})+\|S_{t_0}-T_{t_0}\|\big)\|z_0-x_0\|+ \| (S_{t_0} -
T_{t_0}) x_0 \|
$$
we obtain \eqref{per62}. 

For every $t\in\br_+$,  one has $S_{t_0}(S_t z_0)=S_t(S_{t_0}
z_0)=S_t z_0$, and the uniqueness of $z_0$ for $S_{t_0}$ we infer
that $S_t z_0=z_0$. Now assume that $\mathcal S$ has another fixed
point $\tilde z_0\in\ck$. Then $S_{t_0}\tilde z_0=\tilde z_0$ which
yields $\tilde z_0=z_0$. Moreover, since  $\d(S_{t_0})<1$, which by
Theorem \ref{Dob1} yields that $S$ is uniform asymptotically
stable as well. This completes the proof.
\end{proof}

The proved theorem yields that if at certain time $t_0$ for given a
uniform asymptotically stable semigroup $(T_t)$ one can find a
Markov semigroup $(S_t)$ such that  operators $T_{t_0}$ and
$S_{t_0}$ are very close, then $(S_t)$ is also uniform
asymptotically stable.\\

Let us turn our attention to bounds on an error of estimation of
the averages $A_t(\mathcal{T})$. From Theorems \ref{Dob1} and \ref{per1}
we infer the following result.

\begin{theorem}
Let $(X, X_+, \ck, f)$ be an abstract state space and
$\mathcal T=(T_t)_{t\geq 0}$ be $C_0$-Markov semigroup on $X$. If
there exists  $t_0 \in \bn$ such that $\d
(T_{t_0}) <1$, then for every $x\in\ck$ one has
\begin{equation*}
\left\| A_t(\mathcal{T})x -x_0\right\| \leq\bigg(\frac{t_0 (1- \d
(T_{t_0})^{[t/t_0] -1})}{t(1- \d (T_{t_0}))} +  \d
(T_{t_0})^{[t/t_0]} \frac{t - t_0 [t/ t_0]}{t}\bigg)\|x-x_0\|.
\end{equation*}
\end{theorem}

\begin{proof} The condition of the theorem yields that the semigroup
is uniform asymptotically stable, i.e. for any $x\in K$ one has
$T_tx\to x_0$ as $t\to\infty$. Hence, we obtain
\begin{eqnarray*}
\left\| A_t(\mathcal{T})x
-x_0\right\|&\leq&\frac{1}{t}\int_{0}^t\|x_0-T_sx\|ds\\[2mm]
&=&\frac{1}{t}\int_{0}^t\|T_sx_0-T_sx\|ds\\[2mm]
&\leq&\frac{\|x-x_0\|}{t}\int_{0}^t\d(T_s)ds
\end{eqnarray*}
So, applying \eqref{4} to the last expression, one finds the
required assertion.
\end{proof}

\begin{theorem} Let the conditions of Theorem \ref{per1} be satisfied. Then for every $x,z\in\ck$ one has
\begin{eqnarray*}
\left\|A_t(\mathcal{T})x- A_t(\mathcal{S})z \right\|\leq
\bigg(\frac{t_0 (1- \d (T_{t_0})^{[t/t_0] -1})}{1- \d (T_{t_0})} +
\d(T_{t_0})^{[t/t_0]} (t - t_0 [t/
t_0])\bigg)\bigg(\|B\|+\frac{\|x-z\|}{t}\bigg).
\end{eqnarray*}
\end{theorem}

\begin{proof} By means of \eqref{3} we find
\begin{eqnarray*}
\left\|A_t(\mathcal{T})x- A_t(\mathcal{S})z \right\|&\leq&\frac{1}{t}\int_{0}^t\|T_ux-S_uz\|du\\[2mm]
&=&\frac{1}{t}\int_{0}^t\|T_sx_0-T_s(x)\|ds\\[2mm]
&\leq&\frac{\|x-z\|}{t}\int_{0}^t\d(T_s)ds+\frac{\|B\|}{t}\int_{0}^t\bigg(\int_{0}^s\d(T_u)du\bigg)ds\\[2mm]
&\leq&\frac{\|x-z\|}{t}\int_{0}^t\d(T_s)ds+\|B\|\int_{0}^t\d(T_u)du\\[2mm]
&=&\bigg(\frac{\|x-z\|}{t}+\|B\|\bigg)\int_{0}^t\d(T_u)du
\end{eqnarray*}
So, applying \eqref{4} to the last expression, one finds the
required assertion.
\end{proof}

The proved results in this section demonstrate a direct link between the rate of convergence to stationarity and stability to
perturbations which is an essential principle in the theory of perturbations. These are universal results holding for many 
interesting Markov processes even in classical and non-commutative settings.


\section{Uniform mean ergodicity of Cesaro averages of $C_0$-Markov semigroups}

In this section, we are going to establish an analogue of Theorem
\ref{Dob1} for uniform mean ergodic Markov semigroups, and provide
its application.

Before we prove Theorem \ref{UME}, we need the following auxiliary result. 

\begin{lemma}\label{UMELem}
Let $X$ be a Banach space and let $\mathcal T$ be a bounded linear $C_0$-semigroup on $X$ with spectral radius 1. If the Cesaro averages of $\mathcal T$ converge to zero with respect to the operator norm as $n\to\infty$, then 0 is contained in the resolvent set of $A$ where $A$ is the generator of $\mathcal T$. 
\end{lemma}

\begin{proof}
Assume that zero is the spectral value of $A$. Then 0 is the approximate eigenvalue of $A$ and we can find a corresponding approximate eigenvectors $(x_n)_{n\in\mathbb N}$. Then by the equation
$$
\displaystyle 
T_t x_n - x_n = \int_0^t T_s A x_n ds \ \ \ \forall t >0
$$
we have $\left\| T_t x_n - x_n \right\| \leq t \left\| A x_n \right\| $ and
$$
\displaystyle
\lim_{n\to\infty} \left\| T_t x_n - x_n \right\|  =0.
$$

For the Cesaro averages of $\mathcal T$, 1 becomes the approximate eigenvalues of $A_t (\mathcal T)$ with approximate eigenvectors $(x_n)_{n\in\mathbb N}$. It is the contradiction by assumption that $A_t (\mathcal T)$ converge to zero with respect to the operator norm.
\end{proof}

  \begin{theorem}\label{UME}
Let  $(X, X_+, \ck, f)$ be a strong abstract state space and $\mathcal T =(T_t)_{t\geq 0}$
be a  $C_0$-Markov semigroup on $X$. Then the following statements are
equivalent:
\begin{itemize}
\item[(i)] $\mathcal T$ is weakly mean ergodic;

\item[(ii)] There exists   $t_0 \in\br_+$ such
that $\d (A_{t_0}(T))<1$;

\item[(iii)] $\mathcal T$ is uniformly mean ergodic.
\end{itemize}
\end{theorem}

\begin{proof}
By Theorem \ref{Dob} (v) we infer the implication (i) $\Rightarrow $(ii). The implication  (iii)$ \Rightarrow
$(i) is obvious. Let us prove (ii)$\Rightarrow$ (i).

Assume that there exists $t_0\in\br_+$ such that $\d (A_{t_0}(T))<1$. Denote $\rho:=\d (A_{t_0}(T))$. Since $T_t$ is a Markov
operator on $X$ we have
\begin{eqnarray*}
\displaystyle
\left\| A_t(\mathcal T) (I - T_s) \right\| &\leq& \left\|\frac{1}{t} \int_0^t T_u du - \frac{1}{t} \int_0^t T_u T_s du \right\| \\
&=& \left\|  \frac{1}{t} \int_0^t T_u du - \frac{1}{t} \int_0^{t} T_{u+s}  du \right\|   \ \ \ \ \text{use change of variable as } \ u+s = z \\
\\
&=& \left\|  \frac{1}{t} \int_0^t T_u du - \frac{1}{t} \int_s^{t+s} T_z  dz \right\| \\
&\leq&  \frac{1}{t} \int_0^s  \left\|T_u \right\| du + \frac{1}{t}
\int_t^{t+s} \left\| T_u \right\| du  \leq \frac{2s}{t}.
\end{eqnarray*}
Hence, for each $s\in\br_+$ one gets
\begin{eqnarray*}
\left\| A_t(\mathcal T) (I- A_s(\mathcal T)) \right\|& =& \left\|
A_t(\mathcal T) \bigg(\frac{1}{s} \int_{0}^s
(I-T_u) du\bigg) \right\| \\[2mm]
&=& \left\| \frac{1}{s} \int_{0}^{s} A_t(\mathcal T) (I-T_u) du
\right\| \\[2mm]
&\leq & \frac{1}{s} \int_{0}^{s}\| A_t(\mathcal T) (I-T_u)\| du\\[2mm]
&\leq & \frac{1}{s} \int_{0}^{s}\frac{2u}{t}du\\[2mm]
&=&\frac{s}{t}
\end{eqnarray*}
which implies
\begin{equation}\label{dd1}
\d ( A_t(\mathcal T)(I- A_s(\mathcal T))\leq \| A_t(\mathcal T) (I-
A_s(\mathcal T))\|\leq\frac{s}{t}
\end{equation}

From Theorem \ref{Dob}(ii) one finds
\begin{equation*}
|\d ( A_t(\mathcal T) A_{t_0}(\mathcal T)) - \d ( A_t(\mathcal T))| \leq \d ( A_t(\mathcal T)(I -
A_{t_0}(\mathcal T)).
\end{equation*}

The last inequality with  Theorem \ref{Dob}(iii) yields that
\begin{eqnarray}\label{dd2}
\d ( A_t(\mathcal T)(I-A_{t_0}(\mathcal T))&\geq&  \d (A_t(\mathcal T))-\d ( A_t(\mathcal T)
 A_{t_0}(\mathcal T))\nonumber\\[2mm]
&\geq&  \d (A_t(\mathcal T))-\d (A_t(\mathcal T))\d(A_{t_0}(\mathcal T))\nonumber\\[2mm]
&\geq&  (1-\rho)\d (A_t(\mathcal T))
\end{eqnarray}

So, from \eqref{dd1} and \eqref{dd2} we obtain
$$
\d(A_t(\mathcal T))\leq \frac{t_0}{t(1-\rho)}
$$
this means $\lim\limits_{t\to\infty} \d (A_t(\mathcal T))=0$. Hence, taking into account that $X$ is a strong abstract state space, due to Remark \ref{sASS}, one gets (i).

Now, it is enough to prove  (i)$ \Rightarrow$(iii). Assume $\mathcal T$ is weakly mean ergodic. Let $N = ker f$ and we note that $N$ is closed $\mathcal T$-invariant subspace of $X$ and $X / N$ is one-dimensional.  Let ${\mathcal T}_|$ and ${\mathcal T}_/$ denote the subspaces semigroup on $N$ and the quotient semigroup on $X/N$, respectively. Since $N$ is $\mathcal T$-invariant closed subspace of $X$, then ${\mathcal T}_| = ({T_t}_| )_{t\geq 0}$  is strongly continous and the restriction of $A$, $A_|$, becomes the generator of ${\mathcal T}_|$ with domain $D(A_|)= D(A) \cap N$. Moreover the quotient semigroup ${\mathcal T}_/$ has the generator $A_/$ defined by $A_/ = q(x)= q(Ax)$ with the domain $D(A_/)= q(D(A))$ where $q: X \to X/N$ is the canonical quotient map.  

Since $\mathcal T$ is weakly mean ergodic, then $A_t ({\mathcal T}_|)$ converges to zero in norm . Thus by Lemma \ref{UMELem}, 0 is the resolvent set of ${\mathcal T}_|$.

On the other hand, $I- T_t$ is not invertible for each $t \geq 0$. Hence $I - {T_t}_/$ is not invertible. Since $X/N $ is one-dimensional, ${T_t}_/$ acts as the identitiy for each $t$ on $X/N$. Therefore 0 is a first pole of the resolvent of $A_/$. Since 0 is not a spectral value of $A_|$, 0 is a first pole of the resolvent of $A$. Hence 0 is an eigenvalue of $A$.  
\end{proof}

We notice that in \cite{EM2018} an analogous result for discrete semigroup of Markov operators has been proved.

\begin{corollary}\label{per8}
 Let  $(X, X_+, \ck, f)$ be a strong abstract state space and $\mathcal T =(T_t)_{t\geq 0}$
be a  $C_0$-Markov semigroup. If there exist $\rho \in [0,1)$ and
$t_0 \in\br_+$ such that $\d (A_{t_0}(T)) \leq \rho$, then one has
\begin{eqnarray*}
\sup_{x\in \ck}\left\| A_t(\mathcal T)x -x_0 \right\|\leq
\frac{2t_0}{t(1-\rho)},
\end{eqnarray*}
where $x_0$ is a unique fixed point of $\mathcal T$.
\end{corollary}

Now we prove an analogue of Theorem \ref{per6} for
$A_t(\mathcal{T})$.

\begin{theorem}\label{per7}
Let  $(X, X_+, \ck, f)$ be a strong abstract state space and  $\mathcal T =(T_t)_{t\geq
0}$ be a  $C_0$-Markov semigroup. If $\d (A_{t_0}(T)) < 1$ for some
$t_0>0$, then every $C_0$-Markov semigroup  $\mathcal S
=(S_t)_{t\geq 0}$ satisfying $\|A_{t_0}(\mathcal S)-A_{t_0}(\mathcal
T)\|<1-\d(A_{t_0}(\mathcal T))$, is uniformly mean ergodic and has a
unique fixed point $z_0\in\ck$ such that
\begin{eqnarray} \label{per72}
\|x_0-z_0\|\leq \frac{\|A_{t_0}(\mathcal S)-A_{t_0}(\mathcal
T)\|}{1-\d (A_{t_0}(\mathcal T))-\|A_{t_0}(\mathcal
S)-A_{t_0}(\mathcal T)\|},
\end{eqnarray}
here as before, $x_0$ is a unique fixed point of $\mathcal T$.
\end{theorem}

\begin{proof} We follow the argument of the proof of Theorem \ref{per6}. First we prove that the operator $(I-A_{t_0}(\mathcal
S))$ is invertible on $N$ (see \eqref{NN}). Indeed, take
any $x\in N$, then we have
\begin{eqnarray}\label{per71}
\|A_{t_0}(\mathcal S)x\|&\leq&\|(A_{t_0}(\mathcal S)-A_{t_0}(\mathcal T)x\|+\|A_{t_0}(\mathcal T)x\|\nonumber\\[2mm]
&\leq&\bigg(\underbrace{\|(A_{t_0}(\mathcal S)-A_{t_0}(\mathcal T)\|+\d(A_{t_0}(\mathcal T))}_{\rho}\bigg)\|x\|\nonumber\\[2mm]
&=&\rho\|x\|.
\end{eqnarray}
Hence by \eqref{per71} one gets $\|(A_{t_0}(\mathcal S))^{n}x\|\leq
\rho^n\|x\|$ for all $n\in\bn$. Therefore, the series
$\sum_{n}(A_{t_0}(\mathcal S))^{n}x$ converges. Using the standard
technique, one can see that
$$
(I-A_{t_0}(\mathcal S))^{-1}x=\sum_{n}A_{t_0}(\mathcal S)^{n}x
$$
and moreover, $\|(I-A_{t_0}(\mathcal S))^{-1}x\|\leq
\frac{\|x\|}{1-\rho}$, for all $x\in N$. This means that
$(I-A_{t_0}(\mathcal S))$ is invertible on $N$.

It is clear that the equation $A_{t_0}(\mathcal S)z_0=z_0$ with
$z_0\in\ck$ equivalent to $$(I-A_{t_0}(\mathcal
S))(z_0-x_0)=-(I-A_{t_0}(\mathcal S))x_0.$$ Due to
$(I-A_{t_0}(\mathcal S))x_0\in N$ we conclude the last equation has
a unique solution
$$
z_0=x_0-(I-A_{t_0}(\mathcal S))^{-1}((I-A_{t_0}(\mathcal S))x_0).
$$
From the identity
$$
z_0-x_0=A_{t_0}(\mathcal T)(z_0-x_0)+(A_{t_0}(\mathcal
S)-A_{t_0}(\mathcal T))(z_0-x_0)+(A_{t_0}(\mathcal
S)-A_{t_0}(\mathcal T))x_0
$$
and keeping in mind $z_0-x_0\in N$ one finds
\begin{eqnarray}\label{per73}
\|z_0-x_0\|&\leq& \big(\d(A_{t_0}(\mathcal T))+\|A_{t_0}(\mathcal
S)-A_{t_0}(\mathcal T)\|\big)\|z_0-x_0\|\\[2mm]
&&+\|A_{t_0}(\mathcal S)-A_{t_0}(\mathcal T)\|
\end{eqnarray}
which implies \eqref{per72}.

From $A_{t_0}(\mathcal S)(S_tz_0)=S_t(A_{t_0}(\mathcal
S)z_0)=S_tz_0$, and the uniqueness of $z_0$ for $A_{t_0}(\mathcal
S)$ we infer that $S_tz_0=z_0$ for all $t\in\br_+$, i.e. $z_0$ is a
a fixed point of $\mathcal S$. Moreover, due to \eqref{per71} one
concludes that $\d(A_{t_0}(\mathcal S))<1$, which by Theorem
\ref{UME} yields that $\mathcal S$ is uniformly mean ergodic. Hence,
$z_0$ is a unique fixed point for $\mathcal S$ (see Remark
\ref{dd0}). This completes the proof.
\end{proof}

From Corollary \ref{per8} and Theorem \ref{per7} we immediately
obtain the following result.

\begin{corollary} Let the conditions of Theorem \ref{per7} be
satisfied.  Then for every $x,z\in\ck$ one has
\begin{eqnarray*} \label{per72}
\sup_{x,z\in\ck}\|A_{t_0}(\mathcal T)x-A_{t_0}(\mathcal S)z\|&\leq&
\frac{2t_0}{t(1-\d (A_{t_0}(\mathcal T)))}\\[2mm]
&&+\frac{2t_0}{t(1-\d
(A_{t_0}(\mathcal
T))-\|A_{t_0}(\mathcal S)-A_{t_0}(\mathcal T)\|)}\\[2mm]
&&+
 \frac{\|A_{t_0}(\mathcal
S)-A_{t_0}(\mathcal T)\|}{1-\d (A_{t_0}(\mathcal
T))-\|A_{t_0}(\mathcal S)-A_{t_0}(\mathcal T)\|}.
\end{eqnarray*}
\end{corollary}

\begin{remark} We note that all obtained results in this section are new even if  one takes X as predual of
either a von Neumann algebra or a JBW-algebra. Moreover, they give new insight between the rate of convergence to stationary
points of perturbed uniform/weak mean ergodic $C_0$-Markov semigroups.  
\end{remark}

\bigskip
\bigskip

\section{Unique ergodicity and weighted averages of $C_0$-Markov semigroups}

In this section, we study unique ergodicity of $C_0$-Markov
semigroups in terms of weighted averages.

Let $(X,X_+,\ck,f)$ be an abstract state space. We first recall that a $C_0$-Markov
semigroup $\mathcal T=(T_t)_{t\geq 0}$ defined on $X$ is called {\it
uniquely ergodic} if $\mathcal T$ has a unique fixed point belonging
to $\ck$.

\begin{remark}
Let $\mathcal T=(T_t)_{t\geq 0}$ be a $C_0$-Markov semigroup. If for
some $t_0>0$ the operator $T_{t_0}$ has a unique fixed point
$x_0\in\ck$, then $\mathcal T$ is uniquely ergodic. Indeed, due to
$T_tx_0=T_tT_{t_0}x_0=T_{t+t_0}x_0=T_{t_0}T_tx_0$, so the uniqueness
of the fixed point for $T_{t_0}$ and Markovianity of $T_t$ imply that $T_tx_0=x_0$ for all $t\in\br_+$.
\end{remark}

A positive function $b=b(s)$ defined on $\br_+$ is called a
\textit{weight} if $\displaystyle \int_0^t b(s) ds \to \infty$ as
$t\to\infty$.

We say that a weight $b$ belongs to the \textit{class $\cw$} if for every
$s\in\br_+$ one has
\begin{equation}\label{WA1}
\lim_{t\to\infty}\frac{\int_s^t |b(u) - b(u-s)| du}{\int_0^t b(u)
du}= 0.
\end{equation}

Let us investigate some properties of the class $\cw$.

\begin{proposition}\label{cW} The following statements hold:
\begin{itemize}
\item[(i)] Let $b$ be a weight function. Assume that there is $t_0\geq 0$ such that $\int_0^{t_0}b(s)ds\leq M$ (for some $M>0$) and $b(t)$ is not decreasing (resp. increasing) for $t\geq t_0$. Then
$b\in \cw$;

\item[(ii)]  if  $b\in\cw$, then for every $\l\in\br_+$  one has $\l b\in
\cw$;

\item[(iii)]  if  $b,g\in\cw$, then   $b+g\in
\cw$;

\item[(iv)] let $b,g\in\cw$ such that $b,g$ are bounded, and one has 
\begin{equation}\label{cW1}
\inf_{t\in \br_+} g(t)\geq \a>0, \ \ \ \inf_{t\in\br_+} b(t)\geq \g>0.
\end{equation}
Then 
$b\cdot g\in
\cw$.
\end{itemize}
\end{proposition}

\begin{proof}  (i). Assume that $b$ is non decreasing. From 
$$
\lim_{t\to\infty}\int_{0}^tb(u)du=\infty
$$
we obtain (without loss of generality we may assume that $s\leq
t_0$)
\begin{eqnarray*} \frac{\int_s^t |b(u) -
b(u-s)| du}{\int_0^t b(u) du}&\leq&\frac{ 2M}{\int_0^t b(u)
du}+\frac{\int_{t_0}^t (b(u) - b(u-s)) du}{\int_0^t b(u) du}\\[2mm]
&=&\frac{2M(t_0-s)}{\int_0^t b(u) du}+\frac{\int_{t_0}^t
b(u)du}{\int_0^t b(u) du}-\frac{\int_{t_0-s}^{t-s} b(u)du}{\int_0^t
b(u) du} \to 0  \ \ \textrm{as} \ \ t\to\infty.
\end{eqnarray*}
Hence, \eqref{WA1} holds, which yields the assertion.  Non increasing case can be proceeded by the same argument.

The statement (ii) is obvious. 
Let us prove (iii). Assume that $b,g\in\cw$. Then we have
\begin{eqnarray*}
\frac{\int_s^t |b(u)+g(u) -b(u-s)-g(u-s)| du}{\int_0^t (b(u)+g(u)) du}
&\leq&
\frac{\int_s^t |b(u) -
b(u-s)| du+\int_s^t |g(u) -
g(u-s)| du}{\int_0^t (b(u)+g(u)) du}\\[2mm]
&\leq&
\frac{\int_s^t |b(u) -
b(u-s)| du}{\int_0^t b(u)du}\\[2mm]
&&+\frac{\int_s^t |g(u) -
g(u-s)| du}{\int_0^t g(u) du}\to 0 \ \ \ \textrm{as} \ \ t\to\infty
\end{eqnarray*}
which shows that $b+g\in\cw$.

(iv) From \eqref{cW1} we obtain
$$
\int_0^t b(u) g(u) du\geq \a \int_0^t b(u) du, \ \ \  \int_0^t b(u) g(u) du\geq \g \int_0^t g(u) du.
$$
Hence, one gets
\begin{eqnarray*}
\frac{\int_s^t |b(u)g(u) -b(u-s)g(u-s)| du}{\int_0^t b(u)g(u) du}
&\leq&
\frac{\int_s^t |b(u)||g(u) -g(u-s)| du}{\int_0^t b(u)g(u) du}\\[2mm]
&&+\frac{\int_s^t |b(u) -b(u-s)||g(u-s)| du}{\int_0^t b(u)g(u) du}\\[2mm]
&\leq&
\frac{\sup|b|\int_s^t |g(u) -g(u-s)| du}{\g\int_0^t g(u) du}\\[2mm]
&&+
\frac{\sup|g|\int_s^t |b(u) -b(u-s)| du}{\a\int_0^t b(u) du}\to 0
\end{eqnarray*}
as $t\to\infty$. This means that $b\cdot g\in\cw$. 
\end{proof}

The proved proposition yields that the class is very huge. Now let us provide 
some concrete examples of weight functions $b$ which belongs to
$\cw$.

\begin{example}\label{we}
From Proposition \ref{cW} (i), we immediately find the following weights which belong to the class $\cw$.
\begin{eqnarray*}
&&b(t)=t^\a, \ \ \a>-1;\\[2mm]
&& b(t)=t^\b\ln^{\g}t, \ \ \b,\g\geq 0.
\end{eqnarray*}
\end{example}

For  a given weight function $b$, let us denote
\begin{equation}\label{WA} A_{b,t} (\mathcal T) =
\frac{1}{\int_0^t b(s) ds} \int_0^t b(s) T_s ds
\end{equation}
Clearly, if  $b(t)\equiv 1$, then $A_{b,t} (\mathcal T)=A_{t}
(\mathcal T)$.

A main result of this section is the following one.

\begin{theorem}\label{ATT}
Let  $\mathcal T=(T_t)_{t\geq 0}$ be a $C_0$-Markov semigroup on an
abstract state space $X$ and $x_0$ be a fixed point of $\mathcal T$. Then the following conditions are
equivalent:
\begin{itemize}
\item[(i)] For every $x \in X$ one has
$$ \left\| A_t (\mathcal T) (x) - T_{x_0} x \right\|
\to 0 \ \ \ \textrm{ as} \ \  t\to\infty;$$
\item[(ii)] one has $\displaystyle X = \mathbb C x_0 \oplus \displaystyle \overline{\bigcup\limits_{t\geq 0} (I - T_t)
(X)}$;\\
\item[(iii)] For each $b\in\cw$ and every $x \in X$ one has
$$\displaystyle \left\| A_{b,t} (\mathcal T) (x) - T_{x_0} x
\right\| \to 0\ \ \ \textrm{ as} \ \  t\to\infty.$$
\end{itemize}
Moreover, (iii) implies
\begin{itemize}\item[(iv)] $\mathcal T$ is uniquely ergodic.\\
\end{itemize}

If there is an order unit space $(V,e)$ such that $V^*=X$, then all
(i)-(iv) statements are equivalent.
\end{theorem}

\begin{proof}
The implications (i)$ \Leftrightarrow$ (ii) follows from \cite{EE2}.
Let us consider the implication (ii)$\Rightarrow$ (iii). We have
\begin{eqnarray}\label{1}
\displaystyle
\left\| A_{b,t} (\mathcal T) (I - T_u) \right\| &=& \left\| \frac{1}{\int_0^t b(s) ds} \int_0^t b(s) T_s ds - \frac{1}{\int_0^t b(s) ds} \int_0^t b(s) T_{s+u} ds\right\|  \nonumber \\
&= &  \frac{1}{\int_0^t b(s) ds} \left\| \int_0^t b(s) T_s ds -  \int_u^{t+u} b(s-u) T_s ds\right\|  \nonumber \\
&\leq & \frac{1}{\int_0^t b(s) ds} \left\| \int_0^u b(s) T_s ds -  \int_t^{t+u} b(s-u) T_s ds\right\| \nonumber \\ & & + \frac{1}{\int_0^t b(s) ds} \left\| \int_u^t ( b(s) -  b(s-u)) T_s ds\right\| \nonumber \\
&\leq& \frac{\int_0^u b(s) ds + \int_t^{t+u} b(s-u) ds}{\int_0^t
b(s) ds} + \frac{1}{\int_0^t b(s) ds} \int_u^t |b(s) - b(s-u)| ds.
\end{eqnarray}

Since $b$ is weight and \eqref{WA1} holds, one gets
\begin{equation}\label{2} \displaystyle \frac{1}{\int_0^t b(s) ds}
\int_s^t |b(s) - b(u-s)| ds \to 0, \,\,\, \,\, \frac{1}{\int_0^t
b(s) ds} \int_0^u b(s) ds \to 0 \,\,\, \text{as} \,\, t\to\infty.
\end{equation}

Moreover, one finds
\begin{eqnarray}\label{311}
\frac{1}{\int_0^t b(s) ds} \int_t^{t+u} b(s-u) ds &=& \frac{1}{\int_0^t b(s) ds} \int_{t-u}^{t} b(s) ds \nonumber\\ &=& \frac{1}{\int_0^t b(s) ds}
\bigg( \int_0^{t} b(s) ds - \int_0^{t-u} b(s) ds \bigg) \nonumber \\
&=& 1 - \frac{1}{\int_0^t b(s) ds} \int_0^{t-u} b(s) ds\to 0 \,\,\,
\text{as} \,\,\, t\to\infty.
\end{eqnarray}

Hence, from (\ref{2}) and (\ref{311}), due to (\ref{1}), for any
$u\in\br_+$, we obtain
\begin{eqnarray}\label{31}
\left\| A_{b,t} (\mathcal T) (I - T_u) \right\| \to 0\ \ \ \text{as}
\,\,\, t\to\infty.
\end{eqnarray}

Since $x_0$ is a fixed point of $(T_t)_{t\geq 0}$, due to $X = \bc
x_0 \oplus \overline{\cup_{t\geq 0} (I - T_t) (X)}$, for every $x\in
X$, we obtain
$$
\lim_{t\to\infty}\left\| A_{b,t} (\mathcal T) x - T_{x_0} x
\right\|=0.
 $$

\bigskip

(iii)$ \Rightarrow$ (ii). Let  $\displaystyle \lim_{t\to\infty}
\left\| A_{b,t} (\mathcal T) x - T_{x_0}x \right\| =0$. Take an
arbitrary $x \in X$. Then
\begin{eqnarray*}
\displaystyle
\left\| T_s (T_{x_0} x) - T_{x_0} x \right\| &=& \lim_{t\to\infty} \left\| (T_s - I) A_{b,t} (\mathcal T) x \right\| \\
  & =& \lim_{t\to\infty} \left\|  (T_s -I) \frac{1}{\int_0^t b(u) du} \int_0^t b(u) T_u x du \right\| \\
 &=& \lim_{t\to\infty} \left\| \frac{1}{\int_0^t b(u) du} \int_0^t b(u) ( T_{s+u} x - T_u x) du\right\|  =0
\end{eqnarray*}
for arbitrary $s \geq 0$. So $T_{x_0} x \in \text{Fix}(\mathcal T)$
and also $T_{x_0}^2 x = T_{x_0} T_{x_0} x = T_{x_0} f(x) x_0 =
T_{x_0} x$ for every $x \in X$. Hence $T_{x_0}$ is a continuous
projection onto $\text{Fix} (\mathcal T)$ and since for each $x\in
X$, $T_{x_0} x = f(x) x_0$, we obtain $\text{Fix} (\mathcal T) = \bc
x_0$. Therefore we have $X = \bc x_0 \oplus \text{ker} (T_{x_0})$.
It is enough to prove that
$$
\text{ker}(T_{x_0}) =  \overline{\bigcup_{t\geq 0} (I - T_t) (X)}.
$$

Let $x \in  \overline{\bigcup_{t\geq 0} (I - T_t) (X)}$, then
$$\left\|
T_{x_0} x \right\| = \lim_{t\to\infty} \left\| A_{b,t} (x) \right\|
= \lim_{t\to\infty} \left\| A_{b,t} (\mathcal T)  (I -
T_s) u \right\| =0$$ which means $x \in \text{ker} (T_{x_0})$. For
the other inclusion, assume that there exists $z \in \text{ker}
(T_{x_0}) \setminus \overline{\bigcup\limits_{t\geq 0} (I -
T_t)(X)}$. Then by Hahn-Banach Theorem, we find $g \in X^*$ such
that $ g(z)>0$ and
$g\upharpoonright_{\overline{\bigcup\limits_{t\geq 0} (I - T_t)
(X)}}=0$. The last one yields $g \circ T_t = g$, $\forall t\geq 0$.
From this we obtain
\begin{eqnarray*}
g(T_{x_0}x)&=&g(\lim_{t\to\infty}A_{b,t}x)=\lim_{t\to\infty}g(A_{b,t}x)\\[2mm]
&=&\lim_{t\to\infty} \frac{1}{\int_0^t b(u) du} \int_0^t b(u) g(T_u
x) du\\[3mm]
&=&\lim_{t\to\infty} \frac{1}{\int_0^t b(u) du} \int_0^t b(u) g(
x) du\\[3mm]
&=&g(x)
\end{eqnarray*}
Hence, $g \circ T_{x_0} = g$, this implies $g(z) =g(T_{x_0} z) =0$,
but it contradicts to the choice of  $g$.

\bigskip

Now we prove the implication (iii) $\Rightarrow$ (iv). From the
statement (iii) we have
\begin{eqnarray}\label{32}
\left\| A_{b,t}(\mathcal T)x_0- T_u(A_{b,t}(\mathcal
T)x_0)\right\|\to\|x_0-T_ux_0\| \ \ \text{as} \,\,\, t\to\infty.
\end{eqnarray}
On the other hand, due to \eqref{31} one finds
$$
\left\| A_{b,t}(\mathcal T)x_0- T_u(A_{b,t}(\mathcal
T)x_0)\right\|=\left\| A_{b,t}(\mathcal T)x_0- A_{b,t}(\mathcal
T)(T_ux_0)\right\| \to 0
$$
This with \eqref{32} implies that $T_ux_0=x_0$ for all $u\in\br_+$.
So, $x_0$ is a fixed point of $\mathcal T$. If $\tilde x\in\ck$ is
an other fixed point of $\mathcal T$, then from (iv) we obtain
$$
\|\tilde x-x_0\|=\left\| A_{b,t}(\mathcal T)\tilde x- x_0\right\|\to
0\ \ \ \ \text{as} \,\,\, t\to\infty
$$
Hence, $\tilde x=x_0$, this means $x_0$ is a unique fixed point of
$\mathcal T$, so $\mathcal T$ is uniquely ergodic.\\

 Now let us
assume that $V^*=X$, where $V$ is an order unit space. In this case,
an order unit for $V$ is the functional $f$ which generates the base
of $X$. We want to prove that (iv) implies (i).  Assume that the
semigroup $\mathcal{T}$ is uniquely ergodic, i.e. there is unique
fixed point $x_0$ of $\mathcal{T_*}$. First define a conjugate
semigroup operator $T_{*,t}:V\to V$ as follows:
$$
\langle T_{*,t}u,x\rangle=\langle u, T_tx\rangle, \ \ u\in V, x\in
X.
$$
It is clear that $T_{*,t}f=f$ for all $t\in\br_+$. The obtained
semigroup we denote $\mathcal{T_*}=(T_{*,t})_{t\geq 0}$.

By the standard argument, one can show that for every $u\in
Fix(\mathcal{T_*})\oplus\bigcup\limits_{t\geq 0} (I - T_{*,t})(V)$
we have $A_t(\mathcal{T_*})u\to Pu$, where $Pu=x_0(u)f$ is a
projection onto $Fix(\mathcal{T_*})$. To complete the proof it is
enough to show that $\text{ker} x_0=\overline{\bigcup\limits_{t\geq
0} (I - T_{*,t})(V)}$. It is clear that
$\overline{\bigcup\limits_{t\geq 0} (I - T_{*,t})(V)}\subset
\text{ker} x_0$. To establish the reverse inclusion, we suppose the
contrary, i.e. there is $u_0\in \text{ker}x_0$ such that $u_0\notin
\overline{\bigcup\limits_{t\geq 0} (I - T_{*,t})(V)}$. Then due to
Hahn-Banach theorem, one can find $h\in V^*$ such that $h(u_0)=1$
and $h\upharpoonright_{\overline{\bigcup\limits_{t\geq 0} (I -
T_{*,t})(V)}}=0$. The last one implies $T_th=h$ for all $t\in\br_+$,
i.e. $h$ is a fixed point of $\mathcal{T}$, hence $h=\l x_0$ for
some $\l\in\br$. This implies $h(u_0)=0$, which is a  contradiction.
Therefore,
$$
\lim_{t\to\infty}\|A_t(\mathcal{T_*})u-Pu\|=0, \ \ \ \textrm{for
all} \ \ u\in V.
$$
Hence, due to the duality relation, we infer
$$
\lim_{t\to\infty}\|A_t(\mathcal{T})x-T_{x_0}\|=0, \ \ \ \textrm{for
all} \ \ x\in X.
$$
This completes the proof.
\end{proof}

\begin{corollary} Let  $\mathcal T=(T_t)_{t\geq 0}$ be $C_0$-Markov semigroup on an
abstract state space $X$ such that any condition of Theorem \ref{ATT} is satisfied. Then for every $x \in X$ one has
$$
\lim_{t\to\infty}\frac1t\int_0^t T_{s^2} ds=T_{x_0}.
$$
\end{corollary}
\begin{proof} Let us take a function 
$$
\tilde b(t)=\frac{1}{\sqrt{t}}, \ \ t\in\br_+
$$
which, thanks to Example \ref{we}, is a weight belonging to $\cw$. It easy to see that
$$
\int^t_0\tilde b(s)ds=2\sqrt{t}
$$
Therefore, we have
\begin{eqnarray*}
\frac{1}{\int^t_0\tilde b(s)ds}\int^t_0\tilde b(s)T_sds&=&
\frac{1}{2\sqrt{t}}\int^t_0\frac{1}{\sqrt{s}}T_sds\\[2mm]
&=&\frac{1}{\sqrt{t}}\int^{\sqrt{t}}_0T_{u^2}du.
\end{eqnarray*}

Hence,  Theorem \ref{ATT} implies the  required assertion. 
\end{proof}

\begin{remark} We stress that all obtained results can be adopted to arbitrary Banach spaces by considering a construction given in Example 3 (see Section 2).  Namely, let $X$ be a Banach space over $\br$. Consider a new Banach space
$\tilde X=\br\oplus X$. Let $(T_t)$ be a
$C_0$-semigroup of contractions of $X$ (i.e. $\|T_tx\|\leq\|x\|$ for all $x\in X$, $t\in\br_+$), then the
equality $\tilde T_t(\a,x)=(\a,T_tx)$ defines a $C_0$- Markov semigroup on
$\tilde X$. It is clear that $A_t(\tilde T)(\a,x)=(\a,A_t(T)x)$.
Hence, all obtained results can be applied to the $C_0$-Markov semigroup $(\tilde T_t)$, which allows to produce new types of theorems for
the semigroup $(T_t)$. 

\begin{itemize}
\item[(i)] We note that all obtained results extend the main results
of \cite{Kar,Mit,RKW} to abstract state spaces.

\item[(ii)] In the above construction if we consider the classical $L_p$-spaces, then one may get
the perturbation bounds for uniformly asymptotically stable Markov
chains defined on these $L_p$-spaces. Moreover, if one considers
non-commutative $L_p$-spaces, then the perturbation bounds open new
perspectives into the quantum information theory (see \cite{CRS,RKW,SN}). 
\end{itemize}

\end{remark}

\section*{acknowledgement}
The authors are grateful to professor A. Mitrophanov for his
fruitful discussions and useful suggestions.

\bibliographystyle{amsplain}

\end{document}